\documentclass[12pt]{article}
\usepackage{amsmath,amssymb,amsthm,comment,cite}

\newtheorem{theorem}{Theorem}[section]

\newtheorem{corollary}[theorem]{Corollary}

\def\barr{\begin{array}}
\def\earr{\end{array}}

\title{Generalizing the Lehmer's totient problem}
\author{Marius T\u arn\u auceanu}
\date{October 18, 2021}

\begin{document}

\maketitle

\begin{abstract}
An important unsolved question in number theory is the Lehmer’s totient problem that asks whether there exists any composite number $n$ such that $\varphi(n)\mid n-1$, where $\varphi$ is the Euler's totient function. It is known that if any such $n$ exists, it must be odd, square-free, greater that $10^{30}$, and divisible by at least $15$ distinct primes. Such a number must be also a Carmichael number.

In this short note, we discuss a group-theoretical analogous pro\-blem involving the function that counts the number of automorphisms of a finite group. Another way to generalize the Lehmer’s totient problem is also proposed.
\end{abstract}
\smallskip

{\small
\noindent
{\bf MSC 2020\,:} Primary 20D60, 11A25; Secondary 20D99, 11A99.

\noindent
{\bf Key words\,:} Lehmer's totient problem, Euler's totient function, finite group, group automorphism, exponent of a group.}

\section{Introduction}

The \textit{Euler's totient function} $\varphi$ is one of the most famous functions in number theory. Recall that the totient $\varphi(n)$ of a positive integer $n$ is defined to be the number of positive integers less than or equal to $n$ that are coprime to $n$. In algebra this function is important mainly because it gives the order of the group of units in the ring $(\mathbb{Z}_n,+,\cdot)$. Also, $\varphi(n)$ can be seen as the number of generators or as the number of automorphisms of the cyclic group $(\mathbb{Z}_n,+)$.

Lehmer’s totient problem \cite{6} asks whether the well-known property "$\varphi(n)=n-1\Leftrightarrow n$ is a prime" can be ge\-ne\-ra\-li\-zed to "$\varphi(n)\mid n-1\Leftrightarrow n$ is a prime". This problem has been studied by many mathematicians (see e.g. \cite{2,3,4,6,7}), but up to now no counterexample has been found. Such a counterexample is often called a \textit{Lehmer number}.

We observe that an integer $n\geq 2$ is a prime or a Lehmer number if and only if
\begin{equation}
|G|-1\equiv 0\,\, ({\rm mod}\,\, |{\rm Aut}(G)|),
\end{equation}where $G$ a cyclic group of order $n$. This fact suggests us to consider \textit{arbitrary} finite groups $G$ which satisfy the relation (1). Their description is given by the following theorem.

\begin{theorem}
A finite group $G$ satisfies the relation $(1)$ if and only if it is cyclic and its order is a prime or a Lehmer number.
\end{theorem}

Finally, we indicate another way to extend the Lehmer’s totient problem via group theory.

\bigskip\noindent{\bf Open problem.} Determine the finite groups $G$ satisfying
\begin{equation}
|G|-1\equiv 0\,\, ({\rm mod}\,\, \varphi(G)),
\end{equation} where
\begin{equation}
\varphi(G)=|\{a\in G \mid o(a)=\exp(G)\}|\nonumber
\end{equation}is the generalization of the Euler's totient function studied in \cite{8}.

Note that since $\varphi(\mathbb{Z}_n)=\varphi(n)$ for all $n\in\mathbb{N}^*$, cyclic groups of prime or Lehmer order will be solutions of (2).

\section{Proof of Theorem 1.1}

First of all, we recall the well-known formula for the number of automorphisms of a finite abelian $p$-group (see e.g. \cite{1,5}).

\begin{theorem}
Let $G=\prod_{i=1}^k \mathbb{Z}_{p^{n_i}}$ be a finite abelian$p$-group, where $1\leq n_1\leq n_2\leq...\leq n_k$. Then
\begin{equation}
|{\rm Aut}(G)|=\prod_{i=1}^k (p^{a_i}-p^{i-1})\prod_{u=1}^k p^{n_u(k-a_u)}\prod_{v=1}^k p^{(n_v-1)(k-b_v+1)}\,,
\end{equation}where
\begin{equation}
a_r=max\{s\mid n_s=n_r\} \mbox{ and } b_r=min\{s\mid n_s=n_r\}\,,\, r=1,2,...,k\,.\nonumber
\end{equation}
\end{theorem}

By using Theorem 2.1, we easily get the following corollary.

\begin{corollary}
Let $G$ be a finite abelian $p$-group. If $p\nmid |{\rm Aut}(G)|$, then $G\cong\mathbb{Z}_p$.
\end{corollary}

\begin{proof}
Under the notation in Theorem 2.1, we infer that $n_1=n_2=\dots=n_k=1$ and $a_1=a_2=\dots=a_k=k$. Then the first product in the right side of (3) is $\prod_{i=1}^k (p^k-p^{i-1})$. Obviously, this is not divisible by $p$ if and only if $k=1$. Thus $G\cong\mathbb{Z}_p$, as desired.
\end{proof}

We are now able to prove our main result.

\bigskip\noindent{\bf Proof of Theorem 1.1.} Let $G$ be a finite group satisfying (1).

We first prove that $G$ is abelian. If not, then $Z(G)\neq G$ and so there exists a prime $p$ dividing $|G/Z(G)|$. Since $G/Z(G)$ can be embedded in ${\rm Aut}(G)$, it follows that $p$ divides $|{\rm Aut}(G)|$. Consequently, $p\mid |G|-1$, contradicting the fact that $p\mid |G|$. Thus $G$ is abelian.

Let $G=\prod_{i=1}^m G_i$, where $G_i$ is a finite abelian $p_i$-group, $i=1,2,\dots,m$. Since
\begin{equation}
|{\rm Aut}(G)|=\prod_{i=1}^m |{\rm Aut}(G_i)| \mbox{ and } |G|=\prod_{i=1}^m |G_i|,\nonumber
\end{equation}by (1) we infer that $p_i\nmid |{\rm Aut}(G_i)|$ for each $i$. Then Corollary 2.2 implies $G_i\cong\mathbb{Z}_{p_i}$ and therefore
\begin{equation}
G\cong\prod_{i=1}^m \mathbb{Z}_{p_i}\cong\mathbb{Z}_{p_1p_2\cdots p_m}\nonumber
\end{equation}is cyclic. Moreover, (1) becomes $|G|-1\equiv 0\,\, ({\rm mod}\,\, \varphi(|G|))$, i.e. $|G|$ is a prime or a Lehmer number. This completes the proof.\qed

\vspace*{3ex}\small

\hfill
\begin{minipage}[t]{5cm}
Marius T\u arn\u auceanu \\
Faculty of  Mathematics \\
``Al.I. Cuza'' University \\
Ia\c si, Romania \\
e-mail: {\tt tarnauc@uaic.ro}
\end{minipage}

\end{document}